\documentclass[12pt,reqno]{amsart}

\usepackage[margin=1in]{geometry}
\usepackage{amsmath,amsfonts,amssymb}
\usepackage{graphics,graphicx,color}
\usepackage{cite}
\usepackage{hyperref}
\usepackage{txfonts}

\theoremstyle{plain}
\newtheorem{theorem}{Theorem}[section]
\newtheorem*{theorem*}{Theorem}

\newtheorem{lemma}{Lemma}[section]
\newtheorem{corollary}{Corollary}[section]
\theoremstyle{remark}
\newtheorem{remark}{Remark}[section]

\newcommand{\e}{^\varepsilon}
\newcommand{\eps}{{\varepsilon}}
\newcommand{\ds}{\displaystyle}

\newcommand{\M}{\mathcal{M}}
\newcommand{\U}{\mathrm{U}}

\renewcommand{\a}{\alpha}

\newcommand{\E}{\mathrm{E}}
\renewcommand{\H}{\mathrm{H}}

\newcommand{\cupl}{\bigcup\limits}
\newcommand{\supp}{\mathrm{supp}}
\newcommand{\suml}{\sum\limits}
\newcommand{\intl}{\int\limits}
\newcommand{\liml}{\lim\limits}

\newcommand{\minl}{\min\limits}

\newcommand{\A}{{A}}

\renewcommand{\phi}{\varphi}

\begin{document}

\footnotetext[1]{Department of Mathematics, Karlsruhe Institute of Technology, Germany}

\footnotetext[2]{Mathematical Division, B. Verkin Institute for Low
Temperature Physics and Engineering of the National Academy of
Sciences of Ukraine}

\footnotetext[3]{\noindent Correspondence to: Department of Mathematics, Karlsruhe Institute of Technology, Kaiserstrasse 89-93, Karlsruhe 76133, Germany

 E-mail: andrii.khrabustovskyi@kit.edu, khruslov@ilt.kharkov.ua}

\noindent {\huge {Gaps in the spectrum of the Neumann Laplacian generated by a system of
periodically distributed trap}} \bigskip

\noindent{\large {Andrii Khrabustovskyi\footnotemark[1]\ \footnotemark[2]\ \footnotemark[3],\quad Evgeni Khruslov\footnotemark[2]}}\medskip

\
\\
{\small\textbf{Abstract.} The article deals with a convergence of the spectrum of the Neumann Laplacian in a periodic unbounded domain $\Omega\e$ depending on a small parameter $\eps>0$. The domain has the form $\Omega\e=\mathbb{R}^n\setminus S\e$, where $S\e$ is an $\eps\mathbb{Z}^n$-periodic family of trap-like screens. We prove that for an arbitrarily large $L$  the spectrum has just one gap in $[0,L]$ when $\eps$ small enough, moreover when $\eps\to 0$ this gap converges to some interval whose edges can be controlled by a suitable choice of geometry of the screens. An application to the theory of $2D$-photonic crystals is discussed.
\medskip

}
\\ \\

%=========================================================================
%=========INTRODUCTION====================================================
%=========================================================================
\section*{Introduction}

It is well-known (see, e.g., \cite{Brown,Kuchment,Reed}) that the
spectrum of self-adjoint periodic differential operators has a
band structure, i.e. it is a union of compact intervals called
\textit{bands}.  The neighbouring bands may overlap, otherwise we
have a \textit{gap} in the spectrum (i.e. an open interval that does not
belong to the spectrum but its ends belong to it). In general the existence of
spectral gaps is not guaranteed.

For applications it is interesting to construct the operators with non-void spectral gaps since their presence is important for the description of wave processes which are governed by differential operators
under consideration. Namely, if the wave frequency belongs to a gap, then the corresponding wave cannot propagate in the medium without attenuation. This feature is a dominant requirement for so-called photonic
crystals which are materials with periodic dielectric structure attracting much attention in recent years (see, e.g.,
\cite{Joannopoulos,Dorfler,Kuchment_PC}).

In the present work we derive the effect of opening of spectral gaps for the Laplace operator in $\mathbb{R}^n$ ($n\geq 2$) perforated by a family of periodically distributed traps on which we pose the Neumann boundary conditions. The traps are made from infinitely thin screens (see Fig. \ref{fig1}). In the case $n=2$ this operator describes the propagation of the $H$-polarized electro-magnetic waves in the dielectric medium containing a system of perfectly conducting trap-like screens (see the remark in Section \ref{sec3}).

We describe the problem and main result more precisely. Let $\eps>0$ be a small parameter. Let $S\e=\cupl_{i\in \mathbb{Z}^n}S_i\e$ be a union of periodically distributed screens $S_i\e$ in $\mathbb{R}^n$ ($n\geq 2$). Each screen $S_i\e$ is an $(n-1)$-dimensional surface obtained by removing of a small spherical hole from the boundary of a $n$-dimensional cube. It is supposed that the distance between the screens is equal to $\eps$, the length of their edges is equal to $b\eps$, while the radius of the holes is equal to $d\eps^{n\over n-2}$ if $n>2$ and $e^{-1/d\eps^2}$ if $n=2$. Here $d\in (0,\infty)$,  $b\in (0,1)$ are constants.

\begin{figure}[h]
  \begin{center}
   \begin{picture}(0,0)
    \includegraphics{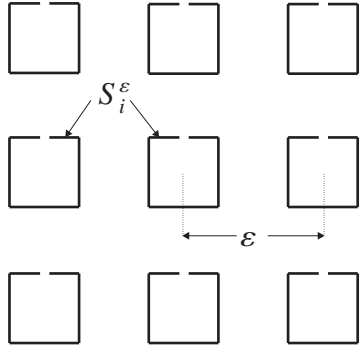}
   \end{picture}
    \setlength{\unitlength}{4144sp}%
    \begin{picture}(2000,2000)(0,0)
      \put(500, 1450){$S_i\e$}
      \put(1350, 600){$\eps$}
    \end{picture}
  \end{center}
  \caption{The system of screens $S_i\e$}\label{fig1}
\end{figure}

By $\mathcal{A\e}$ we denote the Neumann Laplacian in the domain $\mathbb{R}^n\setminus S\e$. Our goal is to describe the behaviour of its spectrum $\sigma(\mathcal{A}\e)$ as $\eps\to 0$.

The main result of this work is as follows (see Theorem \ref{th1}). For an arbitrarily large $L$ the spectrum $\sigma(\mathcal{A}\e)$ has just one gap in $[0,L]$ when $\eps$ is small enough. When $\eps\to 0$ this gap converges to some interval
$(\sigma,\mu)$ depending in simple manner on the coefficients $d$ and $b$. Moreover (see Corollary \ref{cor})  with a suitable choice of $d$ and $b$ this interval can be made equal to an arbitrary preassigned interval in $(0,\infty)$. 

The possibility of opening of spectral gaps by means of a periodic perforation was also investigated in \cite{Nazarov}. Here the authors studied the spectrum of the Neumann Laplacian  in  $\mathbb{R}^2$ perforated by $\mathbb{Z}^2$-periodic family of circular holes. It was proved that the gaps open up when the diameter $d$ of holes is close enough to the distance between their centres (which is equal to $1$). However, the structure of the spectrum in \cite{Nazarov} differs essentially from that one in the present paper. Namely, when $d\to 1$ the spectrum converges (uniformly on compact intervals) to a sequence of points.

Various examples of scalar periodic elliptic operators in the entire space with periodic coefficients were presented in \cite{Hoang,Figotin1,Figotin2,Figotin3,DavHar,Hempel,Zhikov,Khrab1,Khrab2,Green,Post}. In these works spectral gaps are the result of high contrast in (some of) the coefficients of the operator.

The outline of the paper is as follows. In Section \ref{sec1} we describe precisely the operator $\mathcal{A}\e$ and formulate the main result of the paper (Theorem \ref{th1}) describing the behaviour of $\sigma(\mathcal{A}\e)$ as $\eps\to 0$. Theorem \ref{th1} is proved in Section \ref{sec2}. Finally, in Section \ref{sec3} on a formal level of rigour we discuss the applications to the theory of  $2D$ photonic crystals.

%=========================================================================
%=========SECTION1=======================================================
%=========================================================================

\section{\label{sec1}Setting of the problem and the main result}

Let $n\in \mathbb{N}\setminus\{1\}$ and let $\eps>0$. We introduce the following sets:
\begin{itemize}
\item $B=\left\{x=(x_1,\dots,x_n)\in \mathbb{R}^n:\ -b/2<x_i<b/2,\ \forall i\right\}$, where $b\in (0,1)$ is a constant.

\item $D\e=\left\{x\in \partial B:\ |x-x^0|<d\e\right\}$, where $x^0=(0,0,\dots,0,b/2)$ and $d\e$ is defined by the following formula:
\begin{gather}\label{d}
d\e=
\begin{cases}
\ds d\eps^{2\over n-2},& n>2,\\
\ds \eps^{-1}\exp\left(-{1\over d\eps^{2}}\right),& n=2.
\end{cases}
\end{gather}
Here $d>0$ is a constant. It is supposed that $\eps$ is small enough so that
$d\e<b/2$

\item $S\e=\partial B\setminus D\e$
\end{itemize}
For $i\in \mathbb{Z}^n$ we set:
\begin{gather}\label{S}
S_i\e=\eps (S\e+i)
\end{gather}
and (see Fig. \ref{fig1})
\begin{gather}\label{Omega}
\Omega\e=\mathbb{R}^n\setminus\left(\cupl_{i\in \mathbb{Z}^n}\ S_i\e\right)
\end{gather}

Now we define precisely the Neumann Laplacian in $\Omega\e$. We
denote by $\eta\e[u,v]$ the sesquilinear form in
${L_{2}(\Omega\e)}$ which is defined by the formula
\begin{gather}\label{form1}
\eta\e[u,v]=\intl_{\Omega\e}\left(\nabla
u,\nabla \bar v\right)dx
\end{gather}
and the definitional domain $\mathrm{dom}(\eta\e)=H^1(\Omega\e)$.
Here $\left(\nabla u,\nabla
\overline{v}\right)=\ds\suml_{k=1}^n{\partial u\over\partial
x_k}{\partial \overline{v}\over\partial x_k}$. The form
$\eta\e[u,v]$ is densely defined closed and positive. Then (see,
e.g., \cite[Chapter 6, Theorem 2.1]{Kato}) there exists the unique
self-adjoint and positive operator $\mathcal{A}\e$ associated with
the form $\eta\e$, i.e.
\begin{gather}\label{eta-a}
(\mathcal{A}\e u,v)_{L_2(\Omega\e)}= \eta\e[u,v],\quad\forall u\in
\mathrm{dom}(\mathcal{A}\e),\ \forall  v\in \mathrm{dom}(\eta\e)
\end{gather}
{It follows from (\ref{eta-a}) that $\mathcal{A}\e u=-\Delta u$ in
the generalized sense. Using a standard regularity theory (see,
e.g, \cite[Chapter 5]{Taylor}) it is easy to show that each
$u\in\mathrm{dom}(\mathcal{A})$ belongs to $H_{loc}^2(\Omega\e)$, furthermore $\ds\left.{\partial u\over\partial
n}\right|_{\Gamma}=0$ for any smooth $\Gamma\subset
\partial\Omega\e$.}

We denote by $\sigma(\mathcal{A}\e)$ the spectrum of
$\mathcal{A}\e$. To describe the behaviour of
$\sigma(\mathcal{A}\e)$ as $\eps\to 0$ we need some additional
notations.

In the case $n>2$ we denote by $\mathrm{cap}(T)$ the capacity of
the disc 
$$T=\left\{x=(x_1,\dots,x_n)\in \mathbb{R}^n:\ |x|<1,\ x_n=0\right\}$$
Recall (see, e.g, \cite{Landkof}) that it is defined by
$$\mathrm{cap}(T)=\inf_w \intl_{\mathbb{R}^n}|\nabla w|^2 dx$$
where the infimum is taken over smooth and compactly supported in
$\mathbb{R}^n$ functions equal to $1$ on $T$.

We set
\begin{gather}\label{sigmamu}
\sigma=\begin{cases}\ds{\mathrm{cap}(T) d^{n-2}\over 4b^n},&n>2,\\
\ds{\pi d\over 2b^2},&n=2,
\end{cases}\qquad \mu={\sigma\over 1-b^n}
\end{gather}
It is clear that $\sigma<\mu$.

The behaviour of $\sigma(\mathcal{A}\e)$ as $\eps\to 0$ is
described by the following theorem.

%=========================================================================
%=========Theorem 1.1=====================================================
%=========================================================================

\begin{theorem}\label{th1}
Let $L$ be an arbitrary number satisfying $L>\mu$. Then the spectrum of the
operator $\mathcal{A}\e$ has the following structure in
$[0,L]$ when $\eps$ is small enough:
\begin{gather}\label{main1}
\sigma(\mathcal{A}\e)\cap [0,L]=[0,L]\setminus
(\sigma\e,\mu\e)
\end{gather}
where  the interval $(\sigma\e,\mu\e)$ satisfies
\begin{gather}\label{main2}
\liml_{\eps\to 0}\sigma\e=\sigma,\quad \liml_{\eps\to 0}\mu\e=\mu
\end{gather}
\end{theorem}

\begin{corollary} \label{cor}
For an arbitrary interval $(\sigma,\mu)\subset (0,\infty)$ 
there is a family $\{\Omega\e\}_{\eps}$ of periodic unbounded domains in $\mathbb{R}^n$ such that for an arbitrary number $L$ satisfying $L>\mu$  the spectrum of the corresponding Neumann Laplacian $\mathcal{A}\e$ has just one gap in $[0,L]$ when $\eps$ is small enough and this gap converges to the interval $(\sigma,\mu)$ as $\eps\to 0$.
\end{corollary}

\begin{proof}
It is easy to see that the map $(d,b)\overset{\eqref{sigmamu}}\mapsto (\sigma,\mu)$  is one-to-one and maps $(0,\infty)\times(0,1)$ onto $\left\{(\sigma,\mu)\in \mathbb{R}^2:\ 0<\sigma<\mu\right\}$. The inverse map is given by
\begin{gather}\label{inverse}
d=\begin{cases}\sqrt[n-2]{4\sigma(1-\sigma\mu^{-1})(\mathrm{cap}(T))^{-1}},&n>2,\\2\sigma\pi^{-1}(1-\sigma\mu^{-1}),&n=2,\end{cases}\qquad b=\sqrt[n]{1-\sigma\mu^{-1}}
\end{gather}
Then the domain $\Omega\e$ considered in Theorem \ref{th1} with $d$ and $b$ being defined by formula \eqref{inverse} satisfies the requirements of the corollary.
\end{proof}

\begin{remark} {\rm It will be easily seen from the proof of Theorem \ref{th1} that the main result remains valid for an arbitrary open domain $B$ which is compactly supported in the unit cube $(-1/2,1/2)^n$ and whose boundary contains an open flat subset on which we choose the point $x^0$. In this case the coefficients $\sigma$ and $\mu$ are defined again by formula \eqref{sigmamu} but with $|B|$ instead of $b^n$ (here by $|\cdot|$ we denote the volume of the domain).}
\end{remark}

\begin{remark}{\rm One can guess that in order to open up $m>1$  gaps we have to place $m$ screens $S_1\e$, $S_2\e$,\dots,$S_m\e$  in the cube $(-1/2,1/2)^n$. However the proof of this conjecture is more complicated comparing with the case $m=1$. We prove it in our next work. Below we only announce the result for the case $m=2$. 

Let $B_1$ and $B_2$ be arbitrary open cuboids which are compactly supported in $(-1/2,1/2)^n$.
On $\partial B_1$ and $\partial B_2$ we choose the points  $x^{1}$ and $x^{2}$ correspondingly. We  suppose that these points do not belong to the edges of cuboids. Let 
$D_1\e$ and $D_2\e$ be open balls with the radii $d_1\e$ and $d_2\e$ and the centres at $x^1$ and $x^2$ correspondingly. Here $d_j\e$, $j=1,2$ are defined by formula \eqref{d} but with $d_j$ instead of $d$ ($d_j>0$ are constants). For $i\in \mathbb{Z}^n$, $j=1,2$ we set  $S_{ij}\e=\eps (S_j\e+i)$ and finally
$$\Omega\e=\mathbb{R}^n\setminus\left(\cupl_{i\in \mathbb{Z}^n,j=1,2}\ S_{ij}\e\right)$$ 
By $\mathcal{A}\e$ we denote the Neumann Laplacian in $\Omega\e$.

We introduce the numbers $\sigma_j$, $j=1,2$ by formula \eqref{sigmamu} with $d_j$ and $|B_j|$ instead of $d$ and $b^n$ correspondingly.  We suppose that $d_j$ are such that the inequality $\sigma_1<\sigma_2$ holds. Finally we define
the numbers $\mu_j$ by the formula
$$\mu_j={1\over 2}\left(\rho_1+\rho_2+\sigma_1+\sigma_2+(-1)^j \sqrt{(\rho_1+\rho_2+\sigma_1+\sigma_2)^2-4(\rho_1\sigma_2+\rho_1\sigma_2+\sigma_1
\sigma_2)}\right)$$
where $\rho_j=\sigma_j |B_j| (1-|B_j|)^{-1}$.  It is not hard to check that $\sigma_1<\mu_1<\sigma_2<\mu_2$. 

Now, let  $L$ be an arbitrary number satisfying $L>\mu_2$ . Then $\sigma(\mathcal{A}\e)$ has just two gaps in $[0,L]$ when $\eps$ is small enough, moreover the edges of these gaps converge to the intervals $(\sigma_j,\mu_j)$ as $\eps\to 0$. Also it is easily to show that the points $\sigma_j,\mu_j$ can be controlled by a suitable choice of the numbers $d_j$ and the cuboids $B_j$.}

\end{remark}

%=========================================================================
%=========SECTION 2=======================================================
%=========================================================================

\section{\label{sec2}Proof of Theorem \ref{th1}}

We present the proof of Theorem \ref{th1} for the case $n\geq 3$
only. For the case $n=2$ the proof is repeated word-by-word with
some small modifications.

In what follows by $C,C_1...$ we denote generic constants that do
not depend on $\eps$.

By $\langle u \rangle_B$ we denote the mean value of the function
$v(x)$ over the domain $B$, i.e. $$\langle u \rangle_B={1\over
|B|}\intl_B u(x)dx$$ 
Recall that by $|B|$ we denote the volume of the domain $B$.

\subsection{\label{ss21}Preliminaries}

We introduce the following sets (see Fig. \ref{fig2}):
\begin{itemize}
\item[] $Y=\left\{x\in \mathbb{R}^n:\ -1/2<x_i<1/2,\ \forall
i\right\}$.

\item[] $Y\e=Y\setminus S\e$.

\item[] $F=Y\setminus \overline B$.

\end{itemize}

Let $\A\e$ be the Neumann Laplacian in $\eps^{-1}\Omega\e$. It is
clear that
\begin{gather}\label{AA}
\sigma(\mathcal{A}\e)=\eps^{-2}\sigma(\A\e)
\end{gather}
$\A\e$ is an $\mathbb{Z}^n$-periodic operator, i.e. ${\A}\e$ commutes
with the translations $u(x)\mapsto u(x+i)$, $i\in \mathbb{Z}^n$.
For us it is more convenient to deal with the operator $\A\e$ since the
external boundary of its period cell is fixed (it coincides with
$\partial Y$).

\begin{figure}[h]
  \begin{center}
   \begin{picture}(0,0)
   \includegraphics{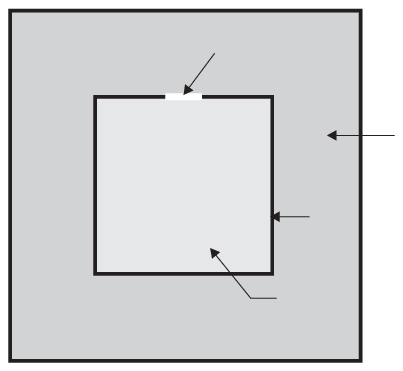}
   \end{picture}
    \setlength{\unitlength}{4144sp}%
    \begin{picture}(1750,2000)(0,0)
      \put(1750, 1000){$F$}
      \put(1350, 625){$S\e$}
      \put(1200, 250){$B$}
      \put(880, 1440){$D\e$}
    \end{picture}
  \end{center}
  \caption{The period cell $Y\e$}\label{fig2}
\end{figure}

In view of the periodicity of $\A\e$ the analysis of the spectrum
$\sigma(\A\e)$ is reduced to the analysis of the spectrum of the
Laplace operator on $Y\e$ with the Neumann boundary conditions on
$S\e$ and so-called \textit{$\theta$-periodic}
boundary conditions on $\partial Y$. Namely, let
$$\mathbb{T}^n=\left\{\theta=(\theta_1,\dots,\theta_n)\in
\mathbb{C}^n:\ \forall k\ |\theta_k|=1\right\}$$ For $\theta\in
\mathbb{T}^n$ we introduce the functional space $H_\theta^1(Y\e)$
consisting of functions from $H^1(Y\e)$ that satisfy the following
condition on $\partial{Y}$:
\begin{gather}\label{theta1}
\forall k=\overline{1,n}:\quad u(x+ e_k)=\theta_k u(x)\text{ for
}x=\underset{^{\overset{\qquad\qquad\uparrow}{\qquad\qquad
k\text{-th place}}\qquad }}{(x_1,x_2,\dots,-{1/2},\dots,x_n)}
\end{gather}
where $e_k={(0,0,\dots,1,\dots,0)}$.

By $\eta^{\theta,\eps}$ we denote the sesquilenear form defined by
formula (\ref{form1}) (with $Y\e$ instead of $\Omega$) and the
definitional domain $H_\theta^1(Y\e)$. We define the operator
$\A^{\theta,\eps}$ as the operator acting in $L_{2}(Y\e)$ and
associated with the form $\eta^{\theta,\eps}$, i.e.
\begin{gather*}
(\A^{\eps,\theta} u,v)_{L_2(Y\e)}= \eta^{\eps,\theta}[u,v],\quad\forall u\in
\mathrm{dom}(\A^{\eps,\theta}),\  \forall v\in \mathrm{dom}(\eta^{\eps,\theta})
\end{gather*}
The functions from $\mathrm{dom}(\A^{\theta,\eps})$ satisfy the Neumann boundary conditions on $S\e$, 
condition (\ref{theta1}) on $\partial Y$ and the condition
\begin{gather}\label{theta2}
\forall k=\overline{1,n}:\quad {\partial u\over\partial x_k}(x+
e_k)=\theta_k {\partial u\over\partial x_k}(x)\text{ for
}x={(x_1,x_2,\dots,-{1/2},\dots,x_n)}
\end{gather}

The operator $\A^{\theta,\eps}$ has purely discrete spectrum. We
denote by
$\left\{\lambda_k^{\theta,\eps}\right\}_{k\in\mathbb{N}}$ the
sequence of eigenvalues of $\A^{\theta,\eps}$ written in the
increasing order and repeated according to their multiplicity.

The Floquet-Bloch theory (see, e.g., \cite{Brown,Kuchment,Reed}) establishes
the following relationship between the spectra of the operators
$\A\e$ and $\A^{\theta,\eps}$:
\begin{gather}\label{repres1}
\sigma(\A\e)=\cupl_{k=1}^\infty L_k\text{, where
}L_k=\cupl_{\theta\in \mathbb{T}^n}
\left\{\lambda_k^{\theta,\eps}\right\}
\end{gather}
The sets $L_k$ are compact intervals.

Also we need the Laplace operators on $Y\e$ with the Neumann boundary
conditions on $S\e$ and either the Neumann or Dirichlet boundary
conditions on $\partial Y=\partial Y\e\setminus S\e$. Namely, we denote by $\eta^{N,\eps}$
(\textit{resp.} $\eta^{D,\eps}$) the sesquilinear form in
$L_2(Y\e)$ defined by formula (\ref{form1}) (with $Y\e$ instead of
$\Omega\e$) and the definitional domain $H^1(Y\e)$ (\textit{resp.}
$\widehat{H}^1_0(Y\e)=\left\{u\in H^1(Y\e):\ u=0\text{ on
}{\partial Y\e\setminus S\e}\right\}$). Then by $\A^{N,\eps}$
(\textit{resp.} $\A^{D,\eps}$) we denote the operator associated
with the form $\eta^{N,\eps}$ (\textit{resp.} $\eta^{D,\eps}$),
i.e.
\begin{gather*}
(\A^{\eps,*} u,v)_{L_2(Y\e)}= \eta^{\eps,*}[u,v],\quad\forall u\in
\mathrm{dom}(\A^{\eps,*}),\ \forall v\in \mathrm{dom}(\eta^{\eps,*})
\end{gather*}
where $*$ is $N$ (\textit{resp.} $D$).

The spectra of the operators $\A^{N,\eps}$ and $\A^{D,\eps}$ are
purely discrete. We denote by
$\left\{\lambda_k^{N,\eps}\right\}_{k\in\mathbb{N}}$
(\textit{resp.}
$\left\{\lambda_k^{D,\eps}\right\}_{k\in\mathbb{N}}$) the sequence
of eigenvalues of $\A^{N,\eps}$ (\textit{resp.} $\A^{D,\eps}$)
written in the increasing order and repeated according to their
multiplicity.

From the min-max principle (see, e.g., \cite[Chapter XIII]{Reed})
and the enclosure $H^1(Y\e) \supset H^1_\theta(Y\e)\supset
\widehat{H}^1_0(Y\e)$ one can easily obtain the inequality
\begin{gather}\label{enclosure}
\forall k\in \mathbb{N},\ \forall\theta\in \mathbb{T}^n:\quad
\lambda_k^{N,\eps} \leq\lambda_k^{\theta,\eps}\leq
\lambda_k^{D,\eps}
\end{gather}

In this end of this subsection we introduce the operators which
will be used in the description of the behaviour of $\lambda_k^N$,
$\lambda_k^D$ and $\lambda_k^\theta$ as $\eps\to 0$. By
$\Delta_F^N$ (\textit{resp.} $\Delta_F^D$, $\Delta_F^\theta$) we
denote the operator which acts in $L_2(F)$ and is defined by the
operation $\Delta$, the Neumann boundary conditions on $\partial B$
and the Neumann (\textit{resp.} Dirichlet, $\theta$-periodic) boundary
conditions on $\partial Y$. By $\Delta_B$ we denote the operator
which acts in $L_2(B)$ and is defined by the operation $\Delta$
and the Neumann boundary conditions on $\partial B$. Finally, we
introduce the operators $\A^N$, $\A^D$, $\A^\theta$ which act in
$L_2(F)\oplus L_2(B)$ and are defined by the following formulae:
\begin{gather*}
\A^N=-\left(\begin{matrix} \Delta_F^{N}&0\\0&\Delta_B
\end{matrix}\right),\quad \A^D=-\left(\begin{matrix}
\Delta_F^{D}&0\\0&\Delta_B\end{matrix}\right),\quad
\A^\theta=-\left(\begin{matrix} \Delta_F^{\theta}&0\\0&\Delta_B
\end{matrix}\right)
\end{gather*}
We denote by $\left\{\lambda_k^{N}\right\}_{k\in\mathbb{N}}$
(\textit{resp.} $\left\{\lambda_k^{D}\right\}_{k\in\mathbb{N}}$,
$\left\{\lambda_k^{\theta}\right\}_{k\in\mathbb{N}}$) the sequence
of eigenvalues of $\A^{N}$ (\textit{resp. }$\A^{D}$,
$\A^{\theta}$) written in the increasing order and repeated
according to their multiplicity. It is clear that
\begin{gather}\label{lambdaN}
\lambda_1^{N}=\lambda_2^{N}=0,\
\lambda_3^{N}>0\\\label{lambdaD} \lambda_1^{D}=0,\
\lambda_2^{D}>0\\\label{lambdaT1}
\lambda_1^{\theta}=\lambda^{\theta}_2=0,\
\lambda_3^{\theta}>0\text{ if
}\theta=(1,1,\dots,1)\\\label{lambdaT2}
\lambda_1^{\theta}=0,\ \lambda_2^{\theta}>0\text{ if
}\theta\not=(1,1,\dots,1)
\end{gather}

\subsection{\label{ss22}Asymptotic behaviour of Dirichlet eigenvalues}
We start from the description of the asymptotic behaviour of the
eigenvalues of the operator $\A^{D,\eps}$ as $\eps\to 0$.

\begin{lemma}\label{lm21}For each $k\in \mathbb{N}$ one has
\begin{gather}
\label{conv1}\liml_{\eps\to 0}\lambda_k^{D,\eps}=\lambda^D_k
\end{gather}
Furthermore
\begin{gather}
\label{conv1+} \lambda^{D,\eps}_1\sim \sigma\eps^2\text{ as
}\eps\to 0
\end{gather} where $\sigma$ is defined by formula
(\ref{sigmamu}).
\end{lemma}

\begin{proof}
We start from the proof of \eqref{conv1}. It is based on the
following abstract theorem.

\begin{theorem*}\textbf{(Iosifyan et al. \cite{IOS})} Let $\mathcal{H}^\eps, \mathcal{H}^0$ be
separable Hilbert spaces, let $\mathcal{L}\e:\mathcal{H}\e\to
\mathcal{H}\e,\ \mathcal{L}^0:\mathcal{H}^0 \to \mathcal{H}^0$ be
linear continuous operators,
$\mathrm{im}\mathcal{L}^0\subset\mathcal{V}\subset \mathcal{H}^0$,
where $\mathcal{V}$ is a subspace in $\mathcal{H}^0$. Suppose that the following conditions $C_1-C_4$ hold:

{$C_1.$} The linear bounded operators $R\e:\mathcal{H}^0\to
\mathcal{H}\e$ exist such that $ \|R\e
f\|^2_{\mathcal{H}\e}\underset{\eps\to 0}\to
\varrho\|f\|^2_{\mathcal{H}^0}$ for any $f\in \mathcal{V}$. Here
$\varrho>0$ is a constant.

{$C_2.$} Operators $\mathcal{L}\e, \mathcal{L}^0$ are positive,
compact and self-adjoint. The norms
$\|\mathcal{L}\e\|_{\mathcal{L}(\mathcal{H}\e)}$ are bounded
uniformly in $\eps$.

{$C_3.$} For any $f\in \mathcal{V}$: $\|\mathcal{L}\e R\e
f-R\e\mathcal{L}^0 f\|_{\mathcal{H}\e}\underset{\eps\to 0}\to 0$.

{$C_4.$} For any sequence $f\e\in \mathcal{H\e}$ such that
$\sup\limits_{\eps} \|f\e\|_{\mathcal{H}\e}<\infty$ the
subsequence $\eps^\prime\subset\eps$ and $w\in \mathcal{V}$ exist
such that $ \|\mathcal{L}\e f\e-R\e
w\|_{\mathcal{H}\e}\underset{\eps=\eps^\prime\to 0}\longrightarrow
0$.

Then for any $k\in\mathbb{N}$\ $$\mu_k\e\underset{\eps\to
0}\to\mu_k$$ where $\{\mu_k\e\}_{k=1}^\infty$ and
$\left\{\mu_k\right\}_{k=1}^\infty$ are the eigenvalues of the
operators $\mathcal{L}\e$ and $\mathcal{L}^0$, which are
renumbered in the increasing order and with account of their
multiplicity.
\end{theorem*}

Let us apply this theorem. We set $\mathcal{H}\e=L_2(Y\e)$,
$\mathcal{H}^0=L_2(F)\oplus L_2(B)$,
$\mathcal{L}\e=(\A^{D,\eps}+\mathrm{I})^{-1}$,
$\mathcal{L}^0=(\A^D+\mathrm{I})^{-1}$ (here $\mathrm{I}$ is the identity operator),
$\mathcal{V}=\mathcal{H}^0$, the operator
$R\e:\mathcal{H}^0\to \mathcal{H}\e$ is defined by the formula
\begin{gather}\label{Reps}
[R\e f](y)=\begin{cases}f_F(y),&y\in F ,\\
f_B(y),&y\in {B},
\end{cases}\quad
f=(f_{F},f_{B})\in \mathcal{H}^0
\end{gather}

Obviously conditions $C_1$ (with $\varrho=1$) and $C_2$ hold (namely, $\|\mathcal{L}\e\|_{\mathcal{L}(\mathcal{H}\e)}\leq 1$).

Let us verify condition $C_3$. Let $f=(f_F,f_B)\in\mathcal{H}^0$. We set
$f\e=R\e f$, $v\e=\mathcal{L}\e f\e$. By $v\e_F$ and $v\e_{B}$ we denote the
restrictions of $v\e$ onto $F$ and $B$ correspondingly. It is clear that
\begin{gather}\label{est1}
\|v_F\e\|^2_{H^1(F)}+\|v_B\e\|^2_{H^1(B)}=\|v\e\|^2_{H^1(Y\e)}\leq 2\|f\e\|^2_{L_2(Y\e)}=2\|f\|_{\mathcal{H}^0}
\end{gather}
We denote
$$\widehat H^1_0(F)=\left\{ w\in H^1(F):\
w|_{\partial Y}=0\right\}$$
Since  $\widehat H_0^1(F)\oplus H^1(B)$ is compactly embedded into $\mathcal{H}^0$ then due to estimate (\ref{est1}) there is a subsequence
$\eps^\prime\subset\eps$ and $v_F\in \widehat H_0^1(F)$, $v_B\in H^1(B)$ such
that
\begin{gather}\label{conv_RB}
\begin{matrix}
v_{F}\e\underset{\eps=\eps^\prime\to 0}\longrightarrow v_{F}
 \text{ weakly in }H^1(F)\text{ and
strongly in }L_2(F)\\
v_{B}\e\underset{\eps=\eps^\prime\to 0}\longrightarrow v_B\text{
weakly in }H^1(B)\text{ and strongly in }L_2({B})
\end{matrix}
\end{gather}

One has the integral equality
\begin{gather}
\label{int_ineq_D}
\intl_{Y\e}\bigg[ \big(\nabla v\e,\nabla
w\e\big)+ v\e w\e-f\e w\e\bigg]dy=0,\quad
\forall w\e\in \widehat H_0^1(Y\e)
\end{gather}
We introduce the set
$$W=\bigg\{(w_F,w_B)\in C^\infty(\overline{F})\oplus C^\infty(\overline{B}):\ w_F|_{\partial Y}=0,\quad \big(\supp (w_F)\cup\supp (w_B)\big)\cap\{x^0\}=\varnothing \bigg\}$$
(here as usual by $\supp(f)$ we denote the \textit{closure} of the set $\{x:\ f(x)\not=0\}$). Let $w=(w_F,w_B)$ be an arbitrary function from $W$.  We set $w\e=R\e w$. It follows from the definition of $W$ that $\supp(w\e)\cap \overline{D\e}=\varnothing$ when $\eps$ is small enough and therefore $w\e\in C^\infty(Y\e)\cap\widehat H_0^1(Y\e)$.

We substitute $w\e$ into
(\ref{int_ineq_D}) and taking into account (\ref{conv_RB}) we pass to the limit in
(\ref{int_ineq_D}) as  $\eps=\eps^\prime\to 0$. As a result we obtain
\begin{gather}\label{int_ineq_final}
\intl_{F}\bigg[\big(\nabla v_{F},\nabla w_F\big)+v_F w_F-f_F
w_F\bigg]dy+\intl_B\bigg[\big(\nabla v_B,\nabla w_B\big)+  v_B
w_B-f_B w_B\bigg]dy=0
\end{gather}
Since $W$ is a dense
subspace of $\widehat H^1_0(F)\oplus H^1(B)$ then (\ref{int_ineq_final}) is valid for an arbitrary $(w_F,w_B)\in\widehat H^1_0(F)\oplus H^1(B)$. It follows from (\ref{int_ineq_final})
that $-\Delta^{D}_F
v_F+v_F=f_F$ and $-\Delta_B v_B+ v_B=f_B$ and consequently
\begin{gather}\label{vA0}
v=\mathcal{L}^0 f,\text{ where }v=(v_F,v_B)
\end{gather}
We remark that $v$ do not depend on the subsequence $\eps^\prime$
and therefore the whole
sequence $(v\e_F,v\e_B)$ converges to $(v_F,v_B)$ as $\eps\to 0$. Condition $C_3$ follows directly from (\ref{Reps}), (\ref{conv_RB}) and (\ref{vA0}). Obviously condition $C_4$  was proved during the proof of $C_3$.

Thus the
eigenvalues $\mu_k\e$ of the operator $\mathcal{L}\e$ converges to
the eigenvalues $\mu_k$ of the operator $\mathcal{L}^0$ as
$\eps\to 0$. But $\lambda_k^{D,\eps}=(\mu_k\e)^{-1}-1$, $
\lambda_k=(\mu_k)^{-1}-1$ that implies (\ref{conv1}).\bigskip

Now we focus on the proof of \eqref{conv1+}. Let $v_1^{D,\eps}$ be the eigenfunction of $\A^{D,\eps}$ that
corresponds to the eigenvalue $\lambda^{D,\eps}_1$ and satisfies
\begin{gather}
\label{normalization}\|v_1^{D,\eps}\|_{L_2(Y\e)}=1\\\label{positivity}
\langle v_1^{D,\eps} \rangle_{B}\geq 0
\end{gather}

One has the following inequalities
\begin{gather} \label{v_est1}
\|v_1^{D,\eps}\|^2_{L_2(F)}\leq C\|\nabla
v_1^{D,\eps}\|^2_{L_2(F)}\\\label{v_est2} \|v_1^{D,\eps}-\langle
v_1^{D,\eps} \rangle_B\|^2_{L_2(B)}\leq C\|\nabla
v_1^{D,\eps}\|^2_{L_2(B)}\\\label{v_est_2+}|B|\cdot|\langle
v_1^{D,\eps} \rangle_B|^2\leq \|v_1^{D,\eps}\|^2_{L_2(B)}\leq 1
\end{gather}
Here the first one is the Friedrichs inequality, the second one is
the Poincar\'{e} inequality and the third one is the Cauchy
inequality. Furthermore one has
\begin{gather}\label{v_est3}
\|\nabla v_1^{D,\eps}\|^2_{L_2(Y\e)}=
\lambda^{D,\eps}_{1}\left(\|v_1^{D,\eps}\|^2_{L_2(F)}+\|v_1^{D,\eps}-\langle
v_1^{D,\eps} \rangle_B\|^2_{L_2(B)}+|B|\cdot|\langle v_1^{D,\eps}
\rangle_B|^2\right)
\end{gather}
Below we will prove (see inequality (\ref{courant}))  that
\begin{gather}\label{v_est4}
\lambda_1^{D,\eps}\leq C\eps^2
\end{gather}
Then it follows from (\ref{normalization}), (\ref{v_est1})-(\ref{v_est4}) that
\begin{gather}\label{v_est5}
\|\nabla
v_1^{D,\eps}\|^2_{L_2(Y\e)}\leq C\eps^2\\\label{v_est6}\|v_1^{D,\eps}\|^2_{L_2(F)}\leq C\eps^2\\
\label{v_est7} \|v_1^{D,\eps}-\langle v_1^{D,\eps}
\rangle_B\|^2_{L_2(B)}\leq C\eps^2
\end{gather}
The last inequality can be specified. Namely, one has
\begin{gather}\label{v_est8}
1=\|v_1^{D,\eps}\|^2_{L_2(F)}+\|v_1^{D,\eps}-\langle
v_1^{D,\eps}\rangle_B\|^2_{L_2(B)}+|B|\cdot|\langle v_1^{D,\eps}
\rangle_B|^2
\end{gather}
Then in view of (\ref{v_est6})-(\ref{v_est8})
$$\left||\langle v_1^{D,\eps}\rangle_B|^2-|B|^{-1}\right|\leq C\eps^2$$ Finally, taking into account \eqref{positivity}
we get:
\begin{gather}\label{v_est9}
\liml_{\eps\to 0}\|v_1^{D,\eps}-|B|^{-1/2}\|^2_{L_2(B)}=0
\end{gather}

Now we construct a convenient approximation
$\mathbf{v}_1^{D,\eps}$ for the eigenfunction $v_1^{D,\eps}$. We
consider the following problem
\begin{gather}\label{bvp-}
\Delta\psi=0\text{ in }\mathbb{R}^n\setminus \overline{T}\\\label{bvp}
\psi=1\text{ in }\partial T\\\label{bvp+}
\psi(x)=o(1)\text{ as }|x|\to\infty
\end{gather}
Recall that $T=\left\{x\in \mathbb{R}^n:\ |x|<1,\ x_n=0\right\}$,
obviously $\partial T=\overline{T}$.  It is well-known that this
problem has the unique solution $\psi(x)$ satisfying
$\intl_{\mathbb{R}^n\setminus T}|\nabla\psi|^2 dx<\infty$.
Moreover it has the following properties:
\begin{gather}\label{psi1}
\psi\in
C^\infty(\mathbb{R}^n\setminus\overline{T})\\\label{psi2}
\psi(x_1,x_2,\dots,x_{n-1},x_n)=\psi(x_1,x_2,\dots,x_{n-1},-x_n)\\\label{psi3}\mathrm{cap}(T)=\ds\intl_{\mathbb{R}^n\setminus
T}|\nabla\psi|^2 dx
\end{gather}
The first two properties imply:
\begin{gather}\label{psi4}
{\partial\psi\over\partial x_n}=0\text{ in } \left\{x\in \mathbb{R}^n:\ x_n=0\right\}\setminus \overline T\\\label{psi5}
\left.{\partial\psi\over\partial x_n}\right|_{x_n=+0}+\left.{\partial\psi\over\partial x_n}\right|_{x_n=-0}=0\text{ in } T
\end{gather}
Furthermore the function $\psi(x)$ satisfies the estimate (see,
e.g, \cite[Lemma 2.4]{March}):
\begin{gather}\label{cap_est}
|D^\a\psi(x)|\leq {C|x|^{2-n-\a}}\ \text{ for }|x|>2,\ |\a|=0,1,2
\end{gather}

We define the function $\mathbf{v}_1^{D,\eps}$ by the formula
\begin{gather}\label{hat_v}
\mathbf{v}_1^{D,\eps}(x)=
\begin{cases}
\ds{1\over 2 \sqrt{|B|}}\psi\left({x-x^0\over
d\e}\right)\varphi\left({|x-x^0|\over l}\right),&x\in F\\\ds
{1\over \sqrt{|B|}}-{1\over 2\sqrt{|B|}}\psi\left({x-x^0\over
d\e}\right)\varphi\left({|x-x^0|\over l}\right),&x\in B\cup  D\e
\end{cases}
\end{gather}
where $\varphi:\mathbb{R}\to \mathbb{R}$ is a twice-continuously
differentiable function such that
\begin{gather}\label{varphi}
\varphi(\rho)=1\text{ as }\rho\leq 1/2\text{ and
}\varphi(\rho)=0\text{ as }\rho\geq 1,
\end{gather}
$l$ is an arbitrary constant satisfying
\begin{gather}\label{l}
0<l<{1\over 4}\minl\left\{{1-b},{b}\right\}
\end{gather}
Here we also suppose that $\eps$ is small enough so that
$d\e<l/2$.  It is easy to see that  the constructed function
$\mathbf{v}_1^{D,\eps}(x)$ belongs to $\mathrm{dom}(\A^{D,\eps})$
in view (\ref{bvp}), (\ref{psi1}), (\ref{psi4}), (\ref{psi5}),
(\ref{varphi}), (\ref{l}).

Taking into account (\ref{psi3}), (\ref{cap_est}) we obtain:
\begin{gather}\label{vbf_est1}
\|\nabla
\mathbf{v}_1^{D,\eps}\|^2_{L_2(Y\e)}\sim 4^{-1}\mathrm{cap}(T)d^{n-2}|B|^{-1}\eps^2=\sigma\eps^2\quad
(\eps\to 0) \\
\label{vbf_est1+}\|\Delta\mathbf{v}_1^{D,\eps}\|^2_{L_2(Y\e)}\leq
C\eps^4
\end{gather}
Since $\mathbf{v}_1^{D,\eps}=0$ on $\partial Y$ and
$\big[\mathbf{v}_1^{D,\eps}-{|B|}^{-1/2}\big]_{int}=0$ on $\partial
B\setminus\left\{x:\ |x-x^0|\leq l\right\}$  (here $[\dots]_{int}$ means the value of the function when we approach $\partial
B$ from inside of $B$) we have the following Friedrichs inequalities
\begin{gather}\label{vbf_est2}
\|\mathbf{v}_1^{D,\eps}\|^2_{L_2(F)}\leq C\|\nabla
\mathbf{v}_1^{D,\eps}\|^2_{L_2(F)}\\\label{vbf_est3}
\|\mathbf{v}_1^{D,\eps}-|B|^{-1/2}\|^2_{L_2(B)}\leq C\|\nabla
\mathbf{v}_1^{D,\eps}\|^2_{L_2(B)}
\end{gather}
It follows from (\ref{vbf_est1}), (\ref{vbf_est2}),
(\ref{vbf_est3}) that
\begin{gather}\label{vbf_est4}
\|\mathbf{v}_1^{D,\eps}\|^2_{L_2(Y\e)}\sim 1\quad (\eps\to 0)
\end{gather}

Using the min-max principle (see, e.g., \cite[Chapter XIII]{Reed}) and taking into account
\eqref{vbf_est1}, \eqref{vbf_est4} we get
\begin{gather}\label{courant}
\lambda^{D,\eps}_{1}={\|\nabla v_1^{D,\eps}\|^2_{L_2(Y\e)}}\leq
\ds{\|\nabla
\mathbf{v}_1^{D,\eps}\|^2_{L_2(Y\e)}\over\|\mathbf{v}_1^{D,\eps}\|_{L_{2}
(Y\e)}^2}\sim
{\mathrm{cap}(T)d^{n-2}|B|^{-1}\eps^2}=\sigma\eps^2\quad (\eps\to
0)
\end{gather}

Now let us estimate the difference
\begin{gather}\label{differenceD}
w\e=v_1^{D,\eps}-\mathbf{v}_1^{D,\eps}
\end{gather} One has
\begin{gather*}
\|w\e\|^2_{L_2(Y\e)}\leq
2\left(\|v_1^{D,\eps}\|^2_{L_2(F)}+\|\mathbf{v}_1^{D,\eps}\|^2_{L_2(F)}\right)+
2\left(\|v_1^{D,\eps}-|B|^{-1/2}\|^2_{L_2(B\e)}+\||B|^{-1/2}-\mathbf{v}_1^{D,\eps}\|^2_{L_2(B\e)}\right)
\end{gather*}
and thus in view of (\ref{v_est6}), (\ref{v_est9}),
(\ref{vbf_est1}), (\ref{vbf_est2}), (\ref{vbf_est3}) we get
\begin{gather}\label{w_est1}
\|w\e\|^2_{L_2(Y\e)}\leq C\eps^2
\end{gather}

Substituting the equality $v_1^{D,\eps}=\mathbf{v}_1^{D,\eps}+w\e$ into (\ref{courant})
and integrating by parts we obtain
\begin{gather*}
\|\nabla w\e\|^2_{L_2(Y\e)}\leq 2(\Delta
\mathbf{v}_1^{D,\eps},w\e)_{L_2(Y\e)}+\left({\|\nabla\mathbf{v}_1^{D,\eps}\|^2_{L_2(Y\e)}\over
\|\mathbf{v}_1^{D,\eps}\|^2_{L_2(Y\e)}}-
\|\nabla\mathbf{v}_1^{D,\eps}\|^2_{L_2(Y\e)}\right)
\end{gather*}
and in view of (\ref{vbf_est1}), (\ref{vbf_est1+}),
(\ref{vbf_est4}), (\ref{w_est1}) we conclude that
\begin{gather}\label{w_est2}
\liml_{\eps\to 0}\eps^{-2}\|\nabla w\e\|^2_{L_2(Y\e)}=0
\end{gather}

Finally using (\ref{vbf_est1}), (\ref{w_est2}) we obtain
\begin{gather}\label{dlast}
\lambda^{D,\eps}_{1}\sim \ds{\|\nabla
\mathbf{v}_1^{D,\eps}\|^2_{L_2(Y\e)}}\sim\sigma\eps^2\quad
(\eps\to 0)
\end{gather}
The lemma is proved.
\end{proof}

\subsection{\label{ss23}Asymptotic behaviour of Neumann eigenvalues}
In this subsection we study the behaviour of the eigenvalues of
the operator $\A^{N,\eps}$ as $\eps\to 0$.

\begin{lemma}\label{lm23}For each $k\in \mathbb{N}$ one has
\begin{gather}
\label{conv2}\liml_{\eps\to 0}\lambda_k^{N,\eps}=\lambda^N_k
\end{gather}
Furthermore
\begin{gather}\label{conv2+}
\lambda^{N,\eps}_2\sim \mu\eps^2\text{ as }\eps\to 0
\end{gather}
where $\mu$ is defined by formula (\ref{sigmamu}).
\end{lemma}

\begin{proof}
The proof of (\ref{conv2}) is similar to the proof of
\eqref{conv1}, so we focus on the proof of \eqref{conv2+}.

Let $v_2^{N,\eps}$ be the eigenfunction of $\A^{N,\eps}$ that
corresponds to $\lambda_2^{N,\eps}$ and satisfies
\begin{gather}
\label{normalization2}\|v_2^{N,\eps}\|_{L_2(Y\e)}=1\\\label{positivity2}
\langle v_2^{N,\eps} \rangle_{B}\geq 0
\end{gather}
Since the eigenspace which corresponds to $\lambda_1^{N,\eps}=0$
consists of constants then
\begin{gather}\label{ort}
(v_2^{N,\eps},1)_{L_2(Y\e)}=0
\end{gather}
and in view of the min-max principle we have
\begin{gather}\label{minN}
\lambda_2^{N,\eps}=\|\nabla
v_2^{N,\eps}\|^2_{L_2(Y\e)}=\minl_{0\not=v\in
H_\#^1(Y\e)}{\|\nabla v\|^2_{L_2(Y\e)}\over\|v\|^2_{L_2(Y\e)}}
\end{gather}
where $H_\#^1(Y\e)=\left\{v\in H^1(Y\e):\ (v\e,1)_{L_2(Y\e)}=0\right\}$

Below we will prove (see inequality (\ref{courant1}))  that
\begin{gather*}
\lambda_2^{N,\eps}\leq C\eps^2
\end{gather*}
Then in the same way as we obtain the estimates for
${v}_1^{D,\eps}$ in Lemma \ref{lm21} we derive the estimates for
${v}_2^{N,\eps}$:
\begin{gather}
\label{vn_est1} \|\nabla{v}_2^{N,\eps}\|^2_{L_2(Y\e)}\leq C\eps^2\\
\label{vn_est2} \|{v}_2^{N,\eps}-\langle
{v}_2^{N,\eps}\rangle_{F}\|^2_{L_2(F)}+\|{v}_2^{N,\eps}-\langle
{v}_2^{N,\eps}\rangle_{B}\|^2_{L_2(B)}\leq C\eps^2
\end{gather}
Using (\ref{vn_est1}), (\ref{vn_est2}) we obtain from
\eqref{normalization2} and \eqref{ort} that
\begin{gather*}
|\langle {v}_2^{N,\eps}\rangle_{F} |^2|F|+ |\langle
{v}_2^{N,\eps}\rangle_{B}|^2|B|\sim 1\quad
(\eps\to 0)\\
\langle {v}_2^{N,\eps}\rangle_{F} |F|+\langle
{v}_2^{N,\eps}\rangle_{B}|B|=0
\end{gather*}
and therefore taking into account  \eqref{positivity2} we get
\begin{gather}\label{262}
\langle {v}_2^{N,\eps}\rangle_{F}\sim -\sqrt{|B|/|F|},\ \langle
{v}_2^{N,\eps}\rangle_{B}\sim\sqrt{|F|/|B|}\quad (\eps\to 0)
\end{gather}

We construct an approximation $\mathbf{v}_2^{N,\eps}$ for the
eigenfunction $v_2^{N,\eps}$ by the following formula:
\begin{gather}\label{vbfn}
\mathbf{v}_2^{N,\eps}(x)=
\begin{cases}
\ds-\sqrt{|B|\over|F|}+{1\over 2
\sqrt{|B||F|}}\psi\left({x-x^0\over
d\e}\right)\varphi\left({|x-x^0|\over l}\right),&x\in F\\\ds
\sqrt{|F|\over|B|}- {1\over 2\sqrt{|B||F|}}\psi\left({x-x^0\over
d\e}\right)\varphi\left({|x-x^0|\over l}\right),&x\in B\cup  D\e
\end{cases}
\end{gather}
Recall that $\psi$ is a solution of \eqref{bvp-}-\eqref{bvp+},
$\varphi:\mathbb{R}\to \mathbb{R}$ is a twice-continuously
differentiable function satisfying \eqref{varphi}, $l$ is a constant satisfying \eqref{l}. One can easily show (taking into account the equality $|B|+|F|=1$)
that $\mathbf{v}_2^{N,\eps}\in \mathrm{dom}(\A^{N,\eps})$.

In the same way as we obtain the estimates for
$\mathbf{v}_1^{D,\eps}$ in Lemma \ref{lm21} we get the estimates
for $\mathbf{v}_2^{N,\eps}$:
\begin{gather}
\label{vbfn_est1} \|\nabla\mathbf{v}_2^{N,\eps}\|^2_{L_2(Y\e)}\sim
{4^{-1}|F|^{-1}|B|^{-1}\mathrm{cap}(T)d^{n-2}\eps^2}=
\mu\eps^2\quad (\eps\to 0)\\
\label{vbfn_est2}\|\Delta\mathbf{v}_2^{D,\eps}\|^2_{L_2(Y\e)}\leq C\eps^4\\
\label{vbfn_est3}
\left\|\mathbf{v}_2^{N,\eps}+\sqrt{|B|/|F|}\right\|^2_{L_2(F)}+
\left\|\mathbf{v}_2^{N,\eps}-\sqrt{|F|/|B|}\right\|^2_{L_2(B)}\leq
C\eps^2
\\\label{vbfn_est4}
\left|\langle\mathbf{v}_2^{N,\eps}\rangle_{Y\e}\right|^2\sim
0\quad (\eps\to 0)\\
\label{vbfn_est5} \|\mathbf{v}_2^{N,\eps}\|^2\sim 1\quad (\eps\to 0)
\end{gather}
Since $\mathbf{v}_2^{N,\eps}-\langle \mathbf{v}_2^{N,\eps}\rangle\in H^1_\#(Y\e)$ then it follows from (\ref{minN}), (\ref{vbfn_est1}), (\ref{vbfn_est4}), (\ref{vbfn_est5}) that
\begin{gather}\label{courant1}
\lambda^{N,\eps}_{2}={\|\nabla v_2^{N,\eps}\|^2_{L_2(Y\e)}}\leq
\ds{\|\nabla
\mathbf{v}_2^{N,\eps}\|^2_{L_2(Y\e)}\over\|\mathbf{v}_2^{N,\eps}-\langle
\mathbf{v}_2^{N,\eps}\rangle_{Y\e}\|_{L_{2}
(Y\e)}^2}\sim\mu\eps^2\quad (\eps\to 0)
\end{gather}

Now let us estimate the difference
\begin{gather}\label{differenceN}
w\e=v_2^{N,\eps}-\mathbf{v}_2^{N,\eps}
\end{gather} One has
\begin{multline*}
\|w\e\|^2_{L_2(Y\e)}\leq
2\left(\left\|v_2^{N,\eps}+\sqrt{|B|/|F|}\right\|^2_{L_2(F)}+
\left\|-\sqrt{|B|/|F|}-\mathbf{v}_2^{N,\eps}\right\|^2_{L_2(F)}\right)+\\+
2\left(\left\|v_2^{N,\eps}-\sqrt{|F|/|B|}\right\|^2_{L_2(B\e)}+
\left\|\sqrt{|F|/|B|}-\mathbf{v}_2^{N,\eps}\right\|^2_{L_2(B\e)}\right)
\end{multline*}
and thus in view of \eqref{vn_est2}, \eqref{262}, \eqref{vbfn_est3}  we get
\begin{gather}\label{wn_est1}
\liml_{\eps\to 0}\|w\e\|^2_{L_2(Y\e)}=0
\end{gather}

Substituting the equality $v_2^{N,\eps}=\mathbf{v}_2^{N,\eps}+w\e$  into (\ref{courant1})
and integrating by parts we get
\begin{gather*}
\|\nabla w\e\|^2_{L_2(Y\e)}\leq 2(\Delta
\mathbf{v}_2^{N,\eps},w\e)_{L_2(Y\e)}+\left({\|\nabla
\mathbf{v}_2^{N,\eps}\|^2_{L_2(Y\e)}\over\|\mathbf{v}_2^{N,\eps}-\langle
\mathbf{v}_2^{N,\eps}\rangle_{Y\e}\|_{L_{2}
(Y\e)}^2}-
\|\nabla\mathbf{v}_2^{N,\eps}\|^2_{L_2(Y\e)}\right)
\end{gather*}
and then in view of (\ref{vbfn_est1}), (\ref{vbfn_est2}),
\eqref{vbfn_est4}, (\ref{vbfn_est5}), (\ref{wn_est1})  we conclude
that
\begin{gather}\label{wn_est2}
\liml_{\eps\to 0}\eps^{-2}\|\nabla w\e\|^2_{L_2(Y\e)}=0
\end{gather}

Finally using (\ref{vbfn_est1}), (\ref{wn_est2})   we obtain
\begin{gather}\label{nlast}
\lambda^{N,\eps}_{2}\sim \ds{\|\nabla
\mathbf{v}_2^{N,\eps}\|^2_{L_2(Y\e)}}\sim\mu\eps^2\quad (\eps\to
0)
\end{gather}
The lemma is proved.
\end{proof}

\subsection{\label{ss23}Asymptotic behaviour of $\theta$-periodic eigenvalues}

To complete the proof of Theorem \ref{th1} we have study the
behaviour of the spectrum of the operator $\A^{\theta,\eps}$.

\begin{lemma}\label{lm24}For each $\theta\in \mathbb{T}^n$ and $k\in \mathbb{N}$ one has
\begin{gather}
\label{conv3}\liml_{\eps\to
0}\lambda_k^{\theta,\eps}=\lambda^\theta_k
\end{gather}
Furthermore
\begin{gather}\label{conv4} 
\lambda^{\theta,\eps}_2\sim \mu\eps^2\text{\quad if
}\theta=(1,1,\dots,1)
\\\label{conv5} \lambda^{\theta,\eps}_1\sim \sigma\eps^2\text{\quad if
}\theta\not=(1,1,\dots,1)
\end{gather}
\end{lemma}
\begin{proof}
The proof of \eqref{conv3} is similar to the proof of
\eqref{conv1}.

The proof of \eqref{conv4} is similar to the proof of
\eqref{conv2+}. Namely, we approximate the eigenfunction
${v}_2^{\theta,\eps}$ of $\A^{\theta,\eps}$ that corresponds to
$\lambda_2^{\theta,\eps}$ and satisfies $\|v_2^{\theta,\eps}\|_{L_2(Y\e)}=1$, $\langle v_2^{\theta,\eps} \rangle_{B}\geq 0$ by the function
$\mathbf{v}_2^{N,\eps}$ \eqref{vbfn} (since
$\mathbf{v}_2^{N,\eps}$ is constant in the vicinity of $\partial Y$ then
it satisfies \eqref{theta1}-\eqref{theta2} with $\theta=(1,...,1)$). The
justification of the asymptotic equality
\begin{gather*}
\lambda^{\theta,\eps}_{2}\sim\ds{\|\nabla
\mathbf{v}_2^{N,\eps}\|^2_{L_2(Y\e)}},\ \theta=(1,\dots,1)
\end{gather*}
is completely similar to the proof of \eqref{nlast}.

Now, we focus on the proof of \eqref{conv5}. Let
$v_1^{\theta,\eps}$ be the eigenfunction of $\A^{\theta,\eps}$
that corresponds to $\lambda_1^{\theta,\eps}$ and satisfies $
\|v_1^{\theta,\eps}\|_{L_2(Y\e)}=1,\ \langle v_1^{\theta,\eps}
\rangle_{B}\geq 0 $. Since $\theta\not=(1,\dots,1)$ then there
exists $l\in \{1,\dots,n\}$ such that
\begin{gather}\label{thetal}
u(x+ e_l)=\theta_l u(x)\text{ for
}x=\underset{^{\overset{\qquad\qquad\uparrow}{\qquad\qquad
l\text{-th place}}\qquad }}{(x_1,x_2,\dots,-{1/2},\dots,x_n)},\
\text{ where } \theta_l\not= 1
\end{gather}
We denote by $S_l^{\pm}$ the faces of $Y$ which are orthogonal to
the axis $x_l$ that is
$$S_l^\pm=\left\{x\in \partial Y: x_l={\pm}{1\over 2}\right\}$$

We need an additional estimate

\begin{lemma}\label{lm24+}
For any $v\in H^1(F)$ the following inequality holds:
\begin{gather}\label{ineq1}
\left|\langle v\rangle_{S_l^{\pm}}-\langle v\rangle_F\right|^2\leq
C\|\nabla v\|^2_{L_2(F)}
\end{gather}
(here $\langle v \rangle_{S_l^{\pm}}=|S_l^{\pm}|^{-1}\intl_{S_l^{\pm}} v(x)dS_x$, where $dS_x$ is the volume (area) form on $S_l^{\pm}$, $|S_l^{\pm}|=\intl_{S_l^{\pm}} dS_x$).
\end{lemma}

\begin{proof} Let $v$ be an arbitrary function from
$C^\infty(\overline{F})$. One has
\begin{gather}\label{NL}
v(x)-v(x+\alpha e_l)=-\intl_0^\a {d v\over dt} (x+e_l t)dt
\end{gather}
where $x\in S_l^-$, $0<\alpha<(1-b)/2$ (and therefore $x+\alpha
e_l\in F$). 
We denote $\widehat F=\left\{x\in F: -{1/2}<x_n<-b/2\right\}$.  Integrating equality \eqref{NL} by 
$x_1$, $x_2$, ..., $x_{l-1},\ x_{l+1}$, ..., $x_n$  from $-1/2$
to $1/2$ and by $\alpha$ from $0$ to $(1-b)/2$, dividing by $(1-b)/2$ and squaring
 we get
\begin{multline}\label{ineq2}
\left|\langle v\rangle_{S_l^-} -\langle
v\rangle_{\widehat{F}}\right|^2=\\=\left|\intl_{-1/2}^{1/2}\cdot\cdot\cdot \intl_{-1/2}^{1/2}\intl_0^{(1-b)/2}\left(\intl_0^{\alpha}{dv\over dt}(x+e_lt)dt\right) d\a dx_1\dots dx_{l-1}dx_{l+1}\dots dx_n\right|^2
\leq\\\leq C\|\nabla v\|^2_{L_2(F)}
\end{multline}
The
fulfilment of inequality \eqref{ineq2} for $v\in H^1(F)$ follows
from the standard embedding and trace theorems.

It is well-known (see, e.g, \cite[Chapter 4]{March}) that the operator
$\Pi: H^1(F)\to H^1(Y)$ exists such that for an arbitrary $v\in
H^1(F)$ one has
\begin{gather}\label{pi}
\left.\Pi v\right|_F=v,\quad \|\Pi v\|_{H^1(Y)}\leq
C\|v\|_{H^1(F)}
\end{gather}
We have the estimate which follows directly from \cite[Lemma 3.1]{Khrab2}
\begin{gather}\label{ineq3}
\left|\langle v\rangle_{\widehat F}-\langle
v\rangle_F\right|^2\leq C\|\nabla \Pi v\|^2_{L_2(Y)}
\end{gather}
Then inequality \eqref{ineq1} (with $S_l^-$) follows from
\eqref{ineq2}-\eqref{ineq3}. The proof of \eqref{ineq1} with $S_l^+$ is 
similar. Lemma \ref{lm24+} is proved.
\end{proof}

Now using Lemma \ref{lm24+} and \eqref{thetal} we get
\begin{multline}\label{ineq4}
\left|\langle
v_1^{\theta,\eps}\rangle_F\right|^2=\left|1-\theta_l\right|^{-2}\left|\langle
v_1^{\theta,\eps}\rangle_F-\theta_l\langle
v_1^{\theta,\eps}\rangle_F\right|^2\leq\\\leq
2\left|1-\theta_l\right|^{-2}\left(\left|\langle
v_1^{\theta,\eps}\rangle_F-\langle
v_1^{\theta,\eps}\rangle_{S_{+}}\right|^2+ \theta_l^2\left|\langle
v_1^{\theta,\eps}\rangle_{S_{-}}-\langle
v_1^{\theta,\eps}\rangle_F\right|^2\right)\leq C\|\nabla
v_1^{\theta,\eps}\|^2_{L_2(F)}
\end{multline}
It follows from  \eqref{ineq4} and the Poincar\'{e} inequality
that
\begin{gather}\label{ineq5}
\|v_1^{\theta,\eps}\|^2_{L_2(F)}=\|v_1^{\theta,\eps}-\langle
v_1^{\theta,\eps}\rangle_F\|^2_{L_2(F)}+\left|\langle
v_1^{\theta,\eps}\rangle_F\right|^2\cdot |F|\leq C\|\nabla
v_1^{\theta,\eps}\|^2_{L_2(F)}
\end{gather}
Thus similarly to the Dirichlet eigenfunction $v_1^{D,\eps}$ (see
\eqref{v_est1}) the function $v_1^{\theta,\eps}$ satisfies the
Friedrichs inequality in $F$ (although $v_1^{\theta,\eps}\not= 0$ on $\partial Y$!). As for the rest the proof of
\eqref{conv5} repeats word-by-word the proof of \eqref{conv1+}: we
approximate the eigenfunction ${v}_1^{\theta,\eps}$ by the
function $\mathbf{v}_1^{D,\eps}$ \eqref{hat_v} (since
$\mathbf{v}_1^{D,\eps}=0$ in the vicinity of $\partial Y$ then it
satisfies (\ref{theta1})-(\ref{theta2}) with an arbitrary $\theta$) and then prove
the asymptotic equality
\begin{gather}\label{asy}
\lambda^{\theta,\eps}_{1}\sim\ds{\|\nabla
\mathbf{v}_1^{D,\eps}\|^2_{L_2(Y\e)}},\ \theta\not=(1,\dots,1)
\end{gather}
The proof of \eqref{asy} (taking into account \eqref{ineq5}) is completely similar to the proof of
\eqref{dlast}.

Lemma \ref{lm24} is proved.
\end{proof}

\subsection{End of the proof of Theorem \ref{th1}} It follows from
\eqref{AA} and \eqref{repres1} that
\begin{gather}\label{sp}
\sigma(\mathcal{A}\e)=\cupl_{k=1}^\infty [a_k^-(\eps),a_k^+(\eps)]
\end{gather}
where the compact intervals $[a_k^-(\eps),a_k^+(\eps)]$ are
defined by
\begin{gather}\label{interval}
[a_k^-(\eps),a_k^+(\eps)]= \cupl_{\theta\in \mathbb{T}^n}
\left\{\eps^{-2}\lambda_k^{\theta,\eps}\right\}
\end{gather}

We denote $\theta_1=(1,1,\dots,1)$, $\theta_2=-\theta_1$. It
follows from \eqref{enclosure} and \eqref{interval} that
\begin{gather}\label{double1}
\eps^{-2}\lambda_k^{N,\eps}\leq a_k^-(\eps)\leq
\eps^{-2}\lambda_k^{\theta_{1},\eps}\\\label{double2}
\eps^{-2}\lambda_k^{\theta_{2},\eps}\leq a_k^+(\eps)\leq
\eps^{-2}\lambda_k^{D,\eps}
\end{gather}
Obviously if $k=1$ then the left and right-hand-sides of
\eqref{double1} are equal to zero. It follows from \eqref{conv2+},
\eqref{conv4} that in the case $k=2$ they both converge to $\mu$
as $\eps\to 0$, while if $k\geq 3$ they converge to infinity in
view of \eqref{lambdaN}, \eqref{lambdaT1}, \eqref{conv2},
\eqref{conv3}. Thus
\begin{gather}\label{a-}
a_1^-(\eps)=0,\quad \liml_{\eps\to 0}a_2^-(\eps)=\mu,\quad
\liml_{\eps\to 0}a_k^-(\eps)=\infty,\ k=3,4,5\dots
\end{gather}
Similarly in view of \eqref{lambdaD}, \eqref{lambdaT2},
\eqref{conv1}, \eqref{conv1+}, \eqref{conv3}, \eqref{conv5} one
has
\begin{gather}\label{a+}
\liml_{\eps\to 0}a_1^+(\eps)=\sigma,\quad  \liml_{\eps\to
0}a_k^+(\eps)=\infty,\ k=2,3,4\dots
\end{gather}
Then \eqref{main1}-\eqref{main2} follow directly from \eqref{sp},
\eqref{a-}-\eqref{a+}. Theorem \ref{th1} is proved.

\section{\label{sec3} Application to the theory of $2D$ photonic crystals}

In this section we apply the results obtained above to the theory of
$2D$ photonic crystals.

\textit{Photonic crystal} is a dielectric medium with periodic structure
whose main property is that the electromagnetic waves of a certain
frequency cannot propagate in it without attenuation. From the
mathematical point of view it means that the corresponding Maxwell
operator has gaps in its spectrum. We refer to
\cite{Joannopoulos,Dorfler,Kuchment_PC} for more details.

It is known that if the crystal is periodic in two directions and
homogeneous with respect to the third one (so-called $2D$ photonic
crystals) then the analysis of the Maxwell operator reduces to the
analysis of scalar elliptic operators. In the case when dielectric
medium occupies the entire space the rigorous justification of this
reduction was carried out in \cite{Figotin2} (in this work
spectral gaps open up due to a high contrast electric permittivity).
In the current work we consider the dielectric medium with a
periodic family of perfectly conducting trap-like screens embedded
into it. It means that we should supplement the Maxwell equations
by suitable boundary conditions on these screens. In this case the
analysis of the Maxwell operator reduces to the analysis of the
Neumann or Dirichlet Laplacians in a $2$-dimensional domain which
is a cross-section of the crystal along periodicity plane.

In this work we derive this reduction on a formal level of rigour. Then using
Theorem \ref{th1} and Lemma \ref{lmApp} below we open up gaps in
the spectrum of the Maxwell operator.

Let us introduce the following sets in $\mathbb{R}^3$:
$$\mathbf{S}_i\e=\left\{(x_1,x_2,z):\ x=(x_1,x_2)\in S_i\e,\ z\in\mathbb{R}\right\},\ \mathbf{\Omega}\e=\left\{(x_1,x_2,z):\ x=(x_1,x_2)\in \Omega\e,\ z\in\mathbb{R}\right\}$$
where  $S_i\e$ and $\Omega\e$ belong to $\mathbb{R}^2$ and are
defined by (\ref{S}) and (\ref{Omega}) correspondingly. We suppose
that $\mathbf{\Omega}\e$ is occupied by a dielectric medium
while the union of  the screens $\mathbf{S}_i\e$ is occupied by a
perfectly conducting material. It is supposed that the electric
permittivity and the magnetic permeability of the material
occupying $\Omega\e$ are equal to $1$. Then the propagation of
electro-magnetic waves in $\mathbf{\Omega}\e$ is described by the
Maxwell equations
\begin{gather*}
\mathrm{curl}\mathbf{E} =-{\partial \mathbf{H}\over\partial t},\
\mathrm{curl}\mathbf{H} ={\partial \mathbf{E}\over\partial t},\
\mathrm{div}\mathbf{H} = \mathrm{div}\mathbf{E} = 0
\end{gather*}
supplemented by the following boundary conditions on $\cupl_{i\in\mathbb{Z}^2} \mathbf{S}_i\e$:
\begin{gather*}
\mathbf{E}_\tau=0,\ \mathbf{H}_\nu=0
\end{gather*}
Here $\mathbf{E}$ and $\mathbf{H}$ are the electric and magnetic fields, $\mathbf{E}_\tau$ and $\mathbf{H}_\nu$ are the tangential and normal components of $\mathbf{E}$ and $\mathbf{H}$ correspondingly.

Looking for monochromatic waves $\mathbf{E}(x,t) = e^{i\omega t}\E(x),\ \mathbf{H}(x; t) = e^{i\omega t}\H(x)$
we obtain
\begin{gather}\label{Mu}
\M\e\U=\omega \U,\ \mathrm{div}{\H} = \mathrm{div}{\E} = 0\text{
in }\mathbf{\Omega}\e,\quad {\E}_\tau=0,\ {\H}_\nu=0\text{ on }\cupl_{i\in
\mathbb{Z}^n} \mathbf{S}_i\e
\end{gather}
where $\U=(\E,\H)^T$, $\M\e$ is the Maxwell operator:
$$\M\e=i\left(\begin{matrix}
0&  -\mathrm{curl} \\\mathrm{curl}&0
\end{matrix}\right)$$
(the subscript $\eps$ emphasizes that this operator acts on
functions defined in $\Omega\e$). For more precise  definition of the Maxwell operator we refer to \cite{BirmanSolo} (the theory developed in this paper covers, in particular, domains with a "screen-like" boundary).

We are interested on the waves propagated along the plane $z=0$,
i.e. when $(\E,\H)$ depends on $x=(x_1,x_2)$ only:
\begin{gather}\label{x1x2}
\E=\left(E_1(x_1,x_2),E_2(x_1,x_2),E_3(x_1,x_2)\right),\
\H=\left(H_1(x_1,x_2),H_2(x_1,x_2),H_3(x_1,x_2)\right)
\end{gather}
Also we suppose that $\U$ is a Bloch
wave that is
\begin{gather}\label{Bloch}
\exists\theta=(\theta_1,\dots,\theta_n)\in \mathbb{T}^n:\ \U(x+i)=\theta^i\U(x),\ i\in
\mathbb{Z}^n
\end{gather}
(here $\theta^i=\prod\limits_{k=1}^n (\theta_k)^{i_k}$). We
call the set of $\omega$ for which there is $\U=(\E,\H)\not=0$
satisfying \eqref{Mu}-\eqref{Bloch} a \textit{(Bloch) spectrum of
the operator $\mathcal{M}\e$}. We denote it by $\sigma(\M\e)$.

Using (\ref{x1x2}) we can easily rewrite the equality
$\M\e\U=\omega \U$ as
\begin{gather}\label{m_inter2}
-i{\partial H_3\over \partial x_2}=\omega E_1,\quad
i{\partial H_3\over \partial x_1}=\omega E_2,\quad -i\left({\partial H_2\over\partial x_1}-{\partial H_1\over\partial x_2}\right)=\omega E_3\\
\label{m_inter1} i{\partial E_3\over \partial x_2}=\omega
H_1,\quad -i{\partial E_3\over \partial x_1}=\omega H_2,\quad
i\left({\partial E_2\over\partial x_1}-{\partial E_1\over\partial
x_2}\right)=\omega H_3
\end{gather}

Let us show that
\begin{gather}\label{impl}
\omega\in \sigma(\M\e) \Longleftrightarrow
\omega^2\in\sigma(\mathcal{A}_0\e)\cup\sigma(\mathcal{A}\e)
\end{gather}
where $\mathcal{A}\e_0$ and $\mathcal{A}\e$ are, correspondingly,
the Dirichlet and Neumann Laplacians in $\Omega\e$. 
Suppose that $\omega \in \sigma(\M\e)$. If $\omega=0$ then $\omega^2\in\sigma(\mathcal{A}\e)$ otherwise 
we express $H_1$ and $H_2$
(\textit{resp.} $E_1$ and $E_2$) from the first two equalities in
\eqref{m_inter1} (\textit{resp.} (\ref{m_inter2})) and plug them
into the third equality in \eqref{m_inter2} (\textit{resp.}
(\ref{m_inter1})). As a result we get the following equalities on $\Omega\e$:
\begin{gather}\label{eqEH3}
-\Delta E_3=\omega^2 E_3\quad  \text{(\textit{resp. }}-\Delta H_3=\omega^2 H_3\text{)}
\end{gather}
Let  $n=(n_1,n_2,0)$ be the unit normal to $\cup \mathbf{S}_i\e$. Since $\E_\tau=0$ then $E|| n$ and therefore $\E\perp (0,0,1)$ and $\E\perp (-n_2,n_1,0)$. Taking this and the first two equalities in \eqref{m_inter2} into account we obtain the following boundary conditions for $E_3$ and $H_3$ on
$\cupl_{i\in\mathbb{Z}^n}S_i\e$:
\begin{gather}\label{eqEH3_bc}
E_3=0,\quad {\partial H_3\over\partial n}={\partial
H_3\over\partial x_1}n_1+{\partial H_3\over\partial
x_2}n_2=i\omega^{-1}(-E_2 n_1+
\lambda^{\theta,\eps}_2\sim \mu\eps^2\text{\quad if
}\theta=(1,1,\dots,1))=0
\end{gather}
Furthermore it follows from \eqref{m_inter2}-\eqref{m_inter1} that 
\begin{gather}\label{nonzero}
\mathrm{U}\not=0\Longleftrightarrow E_3\not=0\text{ or }H_3\not=0
\end{gather}
It follows from \eqref{eqEH3}-\eqref{nonzero} and \eqref{Bloch}
that $\omega^2$ belongs to the spectrum of either the
Dirichlet or Neumann Laplacian in $\Omega\e$. The converse implication in \eqref{impl} is proved similarly.

\begin{remark}{\rm
Above we have dealt with the space 
$$J=\big\{(\E,\H):\ \mathrm{div}\E=\mathrm{div}\H=0\text{ on }\mathbf{\Omega}\e,\ \E_\tau=\H_\mu=0\text{ on }\cup \mathbf{S}_i\e\text{ and }\eqref{x1x2}-\eqref{Bloch}\text{ holds}\big\}$$
We introduce the following subspaces 
\begin{gather*}
J_E=\{(\E,\H)\in J:\ E_1=E_2=H_3=0\},\quad J_H=\{(\E,\H)\in J:\ H_1=H_2=E_3=0\}
\end{gather*}
Their elements are usually called $E$- and $H$-polarized waves. These subspaces are $L_2$-orthogonal and it is easy to see that each $\U\in J$ can be represented in unique way as
$$\U=\U_E+\U_H\text{, where }\U_E\in J_E,\ \U_H\in J_H$$
Moreover $J_E$ and $J_H$ are invariant subspaces of $\M\e$. Thus $\sigma(\M\e)$ is a union of $\sigma(\M\e|_{J_E})$ ($E$-subspectrum) and $\sigma(\M\e|_{J_H})$ ($H$-subspectrum). We have just shown that the set of squares of points from $E$-subspectrum is just the spectrum of the Dirichlet Laplacian, while  the the set of squares of points from $H$-subspectrum is the spectrum of the Neumann Laplacian.}
\end{remark}

The spectrum of $\mathcal{A}\e$ has been
studied above (Theorem \ref{th1}). Now, we describe the spectrum
of the Dirichlet Laplacian in $\Omega\e$. We define it
via a sesquilinear form $\eta_0\e$ which is defined by formula
\eqref{form1} and the definitional domain
$\mathrm{dom}(\eta_0\e)=H^1_0(\Omega\e)$ (here $H^1_0(\Omega\e)$
is a closure in $H^1(\Omega\e)$ of the set of compactly supported
in $\Omega\e$ functions). The Diriclet Laplacian in $\Omega\e$ (we denote it
$\mathcal{A}_0\e$) is
the operator which is generated by this form, i.e. \eqref{eta-a} holds (with $\mathcal{A}_0\e$, $\eta_0\e$ instead of $\mathcal{A}\e$, $\eta\e$). It turns out that its spectrum goes to infinity 
as $\eps\to 0$, namely the following lemma holds true.

\begin{lemma}\label{lmApp} One has:
\begin{gather}\label{dir}
\min\{\lambda: \lambda\in
\sigma(\mathcal{A}_0\e)\}\underset{{\eps\to 0}}\to \infty
\end{gather}
\end{lemma}

\begin{proof}
For $i\in \mathbb{Z}^n$ we denote $B_i\e=\eps(B+i)$,
$F_i\e=\eps(F+i)$. Let $u\in \mathrm{dom} \mathcal{A}\e_0$. Since
$u=0$ on $S_i\e$ one has the following Friedrichs inequalities:
\begin{gather*}
\|u\|^2_{L_2(B_i\e)}\leq C\eps^{2}\|\nabla
u\|_{L_2(B_i\e)}^2,\quad \|u\|^2_{L_2(F_i\e)}\leq
C\eps^{2}\|\nabla u\|_{L_2(F_i\e)}^2
\end{gather*}
Here $C$ is independent of $u$, $i$ and $\eps$. Summing up these
equalities by $i\in\mathbb{Z}^n$ and then integrating by parts we
get:
\begin{gather}\label{fri}
(\mathcal{A}\e_0 u,u)_{L_2(\Omega\e)}\geq C_1\eps^{-2}
\|u\|^2_{L_2(\Omega\e)}
\end{gather}
It follows from \eqref{fri} (see, e.g., \cite[Chapter 6]{Birman})
that the ray $(-\infty,C_1\eps^{-2})$ belongs to the resolvent set
of $\mathcal{A}\e_0$ that imply \eqref{dir}. The lemma is proved.
\end{proof}

Thus using Theorem \ref{th1}, Lemma \ref{lmApp} and
\eqref{impl} we conclude that the spectrum of the Maxwell operator $\M\e$ has
just two gaps in large finite intervals when $\eps$ is
small enough;  as $\eps\to 0$ these gaps converge to the intervals
$(-\sqrt{\sigma},-\sqrt{\mu})$ and $(\sqrt{\sigma},\sqrt{\mu})$.

%=========================================================================
%=========Bibliography======================================================
%=========================================================================

\end{document}